\documentclass[10pt,a4paper]{article}
\usepackage[top=1in,bottom=1in,left=1in,right=1in]{geometry}
\usepackage[fleqn]{amsmath}
\usepackage{amssymb,color,amsthm}
\usepackage[center]{titlesec}
\usepackage{float}
\usepackage{graphicx}
\usepackage{epstopdf}
\usepackage{booktabs}
\usepackage{longtable}
\usepackage{rotating}
\usepackage{multirow}
\usepackage{indentfirst}
\usepackage{titlesec}
\usepackage{bm}
\usepackage{bbm}
\usepackage{url,colortbl}
\usepackage{subfigure}
\usepackage{appendix,graphicx}
\usepackage{array,booktabs,setspace}
\usepackage{booktabs}

\numberwithin{equation}{section}
\graphicspath{{FIG/}}
\newtheorem{Theorem}{Theorem}[section]

\newtheorem{Lemma}{Lemma}[section]
\newtheorem{Remark}{Remark}[section]

\newcounter{RomanNumber}

\makeatletter
\renewcommand{\section}{\@startsection{section}{1}{0mm}
  {-\baselineskip}{0.5\baselineskip}{\bf\leftline}}
\renewcommand{\subsection}{\@startsection{subsection}{1}{0mm}
  {-\baselineskip}{0.5\baselineskip}{\leftline}}
\makeatother

\title{\large Mass- and energy-preserving exponential Runge-Kutta methods for the nonlinear Schr\"odinger equation}
\author{
\small Jin Cui$^{1}$, Zhuangzhi Xu$^2$, Yushun Wang$^2$  and Chaolong Jiang$^{3,4}$\thanks{Corresponding author. E-mail: chaolong\_jiang@126.com.}\\
\small $^1$Department of Basic Sciences, Nanjing Vocational College of Information Technology,\\
\small Nanjing 210023, China \\
\small  $^2$Jiangsu Key Laboratory for Numerical Simulation of Large Scale Complex Systems,\\
\small  School of Mathematical Sciences, Nanjing Normal University, Nanjing 210023, China\\
\small  $^3$Department of Mathematics, College of Liberal Arts and Science,\\
\small National University of Defense Technology, Changsha, 410073, PR China\\
\small $^4$School of Statistics and Mathematics,  Yunnan University of Finance and Economics,\\
\small Kunming 650221, China
}

\begin{document}

\date{}
\maketitle
\titleformat*{\section}{\centering}{\Large}
\titleformat*{\subsection}{\centering}{\large}

\begin{center}{Abstract} \end{center}

\indent In this paper, a family of arbitrarily high-order structure-preserving  exponential Runge-Kutta methods are developed for the nonlinear  Schr\"odinger equation by combining the scalar auxiliary variable approach with the exponential Runge-Kutta method. By introducing an auxiliary variable, we first transform the original model into an equivalent system which admits both mass and modified energy conservation laws. Then applying the Lawson method and the symplectic Runge-Kutta method in time, we derive  a class of mass- and energy-preserving time-discrete schemes which are arbitrarily high-order in time. Numerical experiments are addressed to demonstrate
the accuracy and effectiveness of the newly proposed schemes.


\section{Introduction}
The nonlinear Schr\"odinger (NLS) equation is well-known in many fields of physics, such as quantum physics, plasma physics and nonlinear optics.
In this paper, we consider the following cubic NLS equation
 \begin{align}\label{eq1.1}
{\rm i} \partial_t\psi(x,y,t) =-\frac{1}{2}\Delta \psi(x,y,t)+\beta |\psi(x,y,t)|^2\psi(x,y,t), \ (x,y) \in \Omega,  \ 0 < t \leq T,
 \end{align}
subject to the $(l_1,l_2)$-periodic boundary condition
 \begin{align}\label{eq1.2}
\psi(x,y,t)=\psi(x+l_1,y,t), \ \psi(x,y,t)=\psi(x,y+l_2,t), \ (x,y) \in \Omega,  \ 0 < t \leq T,
 \end{align}
 and the initial condition
 \begin{align}\label{eq1.3}
\psi(x,y,0)=\psi_0(x,y), \ (x,y) \in \Omega,
 \end{align}
 where ${\rm i}\!=\!\sqrt{-1}$ is the complex unit, $t$ is the time variable, $x$ and $y$ are the spatial variables, $\psi(x,y,t)$ is the complex-valued wave function, $\Delta$ is the usual Laplace operator, $\beta$ is a given real constant, and ${\psi_0}(x,y)$ is a given  $(l_1,l_2)$-periodic complex-valued function. The initial-periodic boundary value problem~\eqref{eq1.1}-\eqref{eq1.3}
preserves the following mass and energy conservation laws, respectively,
\begin{align}\label{mass-conservation-law}
M(t):=\int_{\Omega}|{\psi}(x,y,t)|^2dxdy \equiv M(0),\ t\ge 0,
\end{align}
and
\begin{align}
H(t):=\int_{\Omega} \bigg [\frac{1}{2}|\nabla {\psi}(x,y,t)|^2 +\frac{\beta}{2} |\psi|^4 \bigg]dxdy \equiv H(0),\ t\ge 0,
\end{align}
where $\Omega=[0,l_1]\!\times\![0,l_2]$ and $|\nabla \psi|^2=|\psi_x|^2+|\psi_y|^2$.

It is well-known that mass and energy conservation laws are fundamental in the development of numerical schemes, and play a crucial role in the numerical analyses of the NLS equation. Thus, during the past decade,
 various conservative numerical methods which can preserve mass or energy are proposed for numerically solving the NLS equation,  such as the Crank-Nicolson method \cite{ABB13,BC13,gongyz17}, Hamiltonian boundary value method \cite{BBCIamc18}, partitioned averaged vector field method \cite{CLW17}, energy-preserving collocation method \cite{LWjcp2015}, discrete variational derivative method \cite{MF01}, prediction-correction method \cite{lixin} and so on. In recent years, there has been an increasing interest in structure-preserving exponential integrators for conservative or dissipative systems, because of their good stability, high accuracy and high efficiency. In Ref. \cite{CCO08fcm}, Celledoni et al. proposed a symmetric energy-preserving exponential integrators for the cubic Schr\"odinger equation by adopting the  symmetric  projection strategy. In Ref. \cite{LWjcp2015}, motivated by the ideas of exponential integrators and discrete gradients, Li and Wu  constructed a structure-preserving exponential scheme for general conservative or dissipative systems, which was thereafter revisited and generalized by Shen et al \cite{SL19jcp}. In Refs. \cite{DJLQsima19,JLQZmc18}, authors developed different energy-stable exponential integrators for gradient flows. More recently, Jiang et al.~\cite{Jiang2020jcp} designed a linearly implicit energy-preserving exponential integrator for the nonlinear Klein-Gordon equation by combining the scalar auxiliary variable (SAV) approach \cite{SXY18} and exponential integrators. For other structure-preserving exponential integrators, readers are referred to Refs. \cite{BM2017sisc,MW17jcp}. However, to our best knowledge, there has been no reference considering structure-preserving exponential schemes for the NLS eqaution, which can inherit the  properties of both mass and energy.

To meet such challenge,  in this paper, we focus on developing arbitrary high-order structure-preserving methods for the NLS eqaution, which can preserve the discrete mass and energy simultaneously. By introducing an auxiliary variable, we first recast the original model into an equivalent system which  admits both mass and modified energy conservation laws. Then a class of high-order semi-discrete exponential Runge-Kutta methods methods are obtained by using the Lawson method~\cite{Lawson1967} and symplectic Runge-Kutta (RK) method in time. We show that the resulting system can rigorously preserve the semi-discrete mass and modified energy, simultaneously. Numerical tests are presented to verify the theoretical analysis.

\section{Model reformulation}
For simplicity of notations, we define the $L^2$ inner product and its norm as $(f,g)=\int_{{\Omega}}^{}f {\bar g}d{\bf x}$ and $\Vert f \Vert=\sqrt{(f,f)},
\ \forall f,g \in L^2({\Omega})$, respectively, where $\bar{g}$ represents the conjugate of $g$. Denote the linear part of \eqref{eq1.1} as ${\mathcal L}\psi=-\frac{1}{2} \Delta \psi$ for simplicity, we then utilize the SAV idea to derive a SAV reformulation, by introducing an auxiliary variable
\begin{align*}
q(t)=\sqrt{(\psi^2,\psi^2)+C_0},
\end{align*}
where $C_0>0$ to make $q$ well-defined for all $\psi$. The energy functional can be rewritten as the following quadratic
form
\begin{align}\label{modified-energy-conservation-law}
E({t})=( {\mathcal L}\psi, \psi)+\frac{\beta}{2}q^2-\frac{\beta}{2}C_0.
\end{align}
Subsequently, according to the energy variational principle, the original system \eqref{eq1.1} is equivalent to the following SAV reformulated system
\begin{align}\label{eq2.1}
\left\{
\begin{aligned}
&\partial_t \psi=-\text{i}\bigg ( {\mathcal L}\psi+ \frac{\beta|\psi|^2\psi q}{\sqrt{(\psi^2,\psi^2)+C_0}} \bigg), \\
&\frac{d}{dt} q=2 {\rm Re}\bigg(\partial_t \psi,\frac{|\psi|^2\psi}{\sqrt{(\psi^2,\psi^2)+C_0}} \bigg),
\end{aligned}
\right.
\end{align}
with the consistent initial conditions
\begin{align}\label{eq2.6}
\psi(x,y,0)=\psi_0(x,y ),  \  q(0)=\sqrt{\big(  \psi^2_0(x,y), \psi^2_0(x,y)) +C_0   },\ (x,y)\in\Omega,
\end{align}
and the periodic boundary condition \eqref{eq1.2}.

The SAV reformulation~\eqref{eq2.1} satisfies the mass conservation law \eqref{mass-conservation-law}  and  the modified energy conservation law \eqref{modified-energy-conservation-law}, respectively. {Since the systems~\eqref{eq1.1} and \eqref{eq2.1}  are identical at a continuous level, the mass and energy conservation laws of the reformulated system \eqref{eq2.1}  stand naturally.}


Following the Lawson transformation~\cite{Lawson1967}, we define the change of variables $u=\rm exp({\rm i}{\mathcal L} t)\psi$, then the system \eqref{eq2.1} becomes
\begin{align}\label{eq2.2}
\left\{
\begin{aligned}
&\partial_tu\!=\!-{\rm i}\exp({\rm i}\mathcal{L}t)\frac{\beta|\exp(-{\rm i}\mathcal{L}t)u|^2\exp(-{\rm i}\mathcal{L}t)u q}{\sqrt{((\exp(-{\rm i}\mathcal{L}t)u)^2,(\exp(-{\rm i}\mathcal{L}t)u)^2)+C_0}}, \\
&\frac{d}{dt} q\!=\!2{\rm Re}\Bigg(\!\!-{\rm i} \exp(-{\rm i}\mathcal{L}t) {\mathcal L}u,\frac{|\exp(-{\rm i}\mathcal{L}t)u|^2\exp(-{\rm i}\mathcal{L}t)u}{\sqrt{\big((\exp(-{\rm i}\mathcal{L}t)u)^2,(\exp(-{\rm i}\mathcal{L}t)u)^2\big)+C_0}}  \Bigg),\\
\end{aligned}
\right.
\end{align}
where the fact $2{\rm Re}\bigg(\! \exp(-{\rm i}\mathcal{L}t)\partial_t u,\frac{|\exp(-{\rm i}\mathcal{L}t)u|^2\exp(-{\rm i}\mathcal{L}t)u}{\sqrt{\big((\exp(-{\rm i}\mathcal{L}t)u)^2,(\exp(-{\rm i}\mathcal{L}t)u)^2\big)+C_0}}  \bigg)\!=\!0$ was used. The system \eqref{eq2.2} further satisfies the following  mass conservation law
\begin{align}\label{E-mass-conservation-law}
\widetilde{M}(t):= (u,u) \equiv \widetilde{M}(0),\ t\ge 0,
\end{align}
 and modified energy conservation law
 \begin{align}\label{E-modified-energy-conservation-law}
\widetilde{E}(t):=( {\mathcal L}u, u)+\frac{\beta}{2}q^2-\frac{\beta}{2}C_0 \equiv \widetilde{E}(0),\ t\ge 0.
\end{align}

\section{Exponential SAV-RK method}
In this section, we further apply the RK method for the system~\eqref{eq2.2} in time. Choose $\tau=\frac{T}{N}$ be the time step, where $N$ is a positive integer number, and denote $t_{n}=n\tau$ for $n=0,1,2\cdots,N$; let $\psi^n$ be the numerical approximation of  $\psi(x,y,t_n)$ for $n=0,1,2,\cdots,N$. Applying a RK method to the system \eqref{eq2.2}, we have
\begin{align}\label{RK-NLSE}
\left\{
\begin{aligned}
&U_i=u^n+\tau\sum_{j=1}^{s}a_{ij}{\widetilde k}_j, \ \ Q_i=q^n+\tau\sum_{j=1}^{s}a_{ij}l_j, \\
&{\widetilde k}_i=-{\rm i}\exp({\rm i}\mathcal{L}(t_n\!+c_i\tau)) \frac{\beta|\exp(-{\rm i}\mathcal{L}(t_n\!+c_i\tau))U_i|^2\exp(-{\rm i}\mathcal{L}(t_n\!+c_i\tau))U_i Q_i}
{\sqrt{\big((\exp(-{\rm i}\mathcal{L}(t_n\!+c_i\tau)) U_i)^2,(\exp(-{\rm i}\mathcal{L}(t_n\!+c_i\tau)) U_i)^2\big)+C_0}} , \\
\\
&l_i=~2{\rm Re}\bigg(-{\rm i}\exp(-{\rm i}\mathcal{L}(t_n\!+c_i\tau)) \mathcal{L}U_i, \frac{|\exp(-{\rm i}\mathcal{L}(t_n\!+c_i\tau))U_i|^2\exp(-{\rm i}\mathcal{L}(t_n\!+c_i\tau))U_i}
{\sqrt{\big((\exp(-{\rm i}\mathcal{L}(t_n\!+c_i\tau)) U_i)^2,(\exp(-{\rm i}\mathcal{L}(t_n\!+c_i\tau)) U_i)^2\big)+C_0}}\bigg), \\
&u^{n+1}=u^n+\tau\sum_{i=1}^{s}b_i{\widetilde k}_i, \ \ q^{n+1}=q^n+\tau\sum_{i=1}^sb_il_i,
\end{aligned}
\right.
\end{align}
where $a_{ij}, b_i, i,j=1,\cdots,s$ are RK coefficients, and $U_i,\ \Psi_i$ and $Q_i$ are numerical approximations of $u(x,y,t_n\!+c_i\tau),\ \psi(x,y,t_n\!+c_i\tau)$ and ${q(t_n\!+c_i\tau)}$, respectively with $c_i=\sum_{j=1}^sa_{ij}$.

After manipulating the exponentials (i.e., $\psi^n=\exp(-{\rm i}\mathcal{L}t_n)u^n$, $\Psi_i=\exp(-{\rm i}\mathcal{L}(t_n\!+c_i\tau))U_i$ and $k_i=\exp(-{\rm i}\mathcal{L}(t_n\!+c_i\tau)){\widetilde k}_i$), the discretization can be rewritten in terms
of the original variable to give a class of exponential Runge-Kutta (ERK) methods for solving \eqref{eq2.1} as follows:
\begin{align}\label{ERK1}
\left\{
\begin{aligned}
&\Psi_i=\exp(-{\rm i}\mathcal{L}c_i\tau) \psi^n+\tau\sum_{j=1}^{s}a_{ij}\exp({\rm i}\mathcal{L}(c_j-c_i)\tau) k_j,\\
&Q_i=q^n+\tau\sum_{j=1}^{s}a_{ij}l_j,
\end{aligned}
\right.
\end{align}
where
$k_i\!=\!-{\rm i}\ \frac{\beta|\Psi_i|^2\Psi_i Q_i}
{\sqrt{\big(\Psi_i^2,\Psi_i^2\big)\!+\!C_0}} ,
l_i\!=2{\rm Re}\!\Big(\!-{\rm i}\mathcal{L}\Psi_i, \frac{|\Psi_i|^2\Psi_i}
{\sqrt{\big(\Psi_i^2,\Psi_i^2\big)\!+\!C_0}}\Big)$.  Then $\psi^{n+1}$ and $q^{n+1}$ are updated by
\begin{align}\label{ERK2}
\left\{
\begin{aligned}
&\psi^{n+1}=\exp(-{\rm i}\mathcal{L}\tau) \psi^n+\tau\sum_{i=1}^{s}b_i \exp(-{\rm i}\mathcal{L}(1-c_i)\tau) k_i,\\
&q^{n+1}=q^n+\tau\sum_{i=1}^sb_il_i,
\end{aligned}
\right.
\end{align}
which is the exponential scalar auxiliary variable Runge-Kutta method (ESAV-RK) method for the NLS equation.

{
\begin{Lemma} \cite{JLQZmc18}\label{lem3.1} For the symmetric positive definite operator $\mathcal{L}$ and the operator $\exp({\rm i}\mathcal{L}t)$, we have the following
results:
\begin{itemize}
\item $\exp({\rm i}\mathcal{L}t)$ commutes with $\mathcal{L}$;
\item $\exp({\rm i}\mathcal{L}t)^{*}=\exp(({\rm i}\mathcal{L}t)^*)=\exp(-{\rm i}\mathcal{L}t)$,
\end{itemize}
\end{Lemma}
\noindent where $\exp({\rm i}\mathcal{L}t)^{*}$ denotes the adjoint operator of $\exp({\rm i}\mathcal{L}t)$.
}

\begin{Theorem} \label{th2.1}  If the coefficients of a RK method  satisfy
\begin{align}\label{symplectic_condition}
b_ia_{ij}+b_ja_{ji}=b_ib_j,\ \forall \ i,j=1,\cdots,s,
\end{align}
 the proposed ERK method \eqref{ERK1}-\eqref{ERK2} can preserve the semi-discrete mass and modified energy conservation laws, respectively, that is,
\begin{align*}
M^{n}=M^0,\ E^{n}=E^0, \ n=1,\cdots,N,
\end{align*}
where
\begin{align}\label{eq3.04}
 M^n=( \psi^n,\psi^n),\ E^{n}=( {\mathcal L}\psi^n, \psi^n)+\frac{\beta}{2}(q^n)^2-\frac{\beta}{2}C_0.
\end{align}
\end{Theorem}
\begin{proof} According to Theorem 2.2 of Ref. \cite{ELW06}, if the coefficients of a RK method  satisfy \eqref{symplectic_condition}, the proposed RK method \eqref{RK-NLSE} satisfies the following semi-discrete mass conservation law
\begin{align*}
\widetilde{M}^n:=(u^n,u^n) \equiv \widetilde{M}^0,\ n=1,2,\cdots,N,
\end{align*}
 and modified energy conservation law
 \begin{align*}
\widetilde{E}^n:=( {\mathcal L}u^n, u^n)+\frac{\beta}{2}(q^n)^2-\frac{\beta}{2}C_0 \equiv \widetilde{E}^0,\ n=1,2,\cdots,N.
\end{align*}
 With {\bf Lemma~\ref{lem3.1}} and $\psi^n=\exp(-{\rm i}\mathcal{L}t_n)u^n$, the above semi-discrete mass and energy conservation laws can be rewritten in terms
of the original variable given by
  \begin{align*}
 M^n:=( \psi^n,\psi^n)\equiv M^0,\ E^{n}:=( {\mathcal L}\psi^n, \psi^n)+\frac{\beta}{2}(q^n)^2-\frac{\beta}{2}C_0\equiv E^0,\ n=1,2,\cdots,N.
\end{align*}
This completes the proof.
  \end{proof}

\begin{Remark} A numerical scheme that preserves both mass and energy conservation laws of the NLS equation after time and spatial
discretizations is known as a mass- and energy-preserving method. Thus, for the spatial discretization, we shall pay special attentions to the following three aspects: \begin{itemize}
\item preserve the symmetric positive definite property of {the operator} $\mathcal{L}$;
\item preserve the discrete integration-by-parts formulae \cite{DO11};
\item is high-order accuracy which is compatible with the time-discrete methods.
\end{itemize}
Based on these statements and the periodic boundary condition, the standard Fourier pseudo-spectral method is chosen for spatial
discretizations which is omitted here due to space limitation. Interested
readers are referred to Refs.~\cite{gongyz17,ST06} for details.
\end{Remark}

\begin{Remark} It is noted that the original discrete Hamiltonian energy at time level $t_n$ is given by
\begin{align} \label{eq3.07}
H^n=( \mathcal L \psi^n,\psi^n )+\frac{\beta}{2}\big((\psi^n)^2,(\psi^n)^2\big).
\end{align}
 However, we should note that the modified energy \eqref{eq3.04} is only equivalent to the Hamiltonian energy \eqref{eq3.07} in the continuous sense,
 but not for the discrete sense. Thus, the proposed schemes cannot preserve the discrete Hamiltonian energy exactly.
\end{Remark}

\section{Numerical examples}
In this section, some numerical examples are presented briefly to demonstrate the accuracy, invariants-preservation, as well as the practicability of
the proposed schemes. For simplicity, in the rest of this paper, we take for example the 4th- and 6th-order Gauss methods, denoted by ESAV-RK4 and ESAV-RK6, respectively. The RK coefficients of the corresponding numerical methods can be found in Ref.~\cite{ELW06}.


The NLS equation \eqref{eq1.1} admits the following progressive plane wave solution
\begin{align*}
&\psi_1(x,y,t)=\exp({\rm i}(k_1x+k_2y-w_1t)), ~~ w_1=(k_1^2+k_2^2)/2+\beta,  ~~ \text{for}~d=2, \\
&\psi_2(x,y,z,t)=\exp({\rm i}(k_1x+k_2y+k_3z-w_2t)), ~~w_2=(k_1^2+k_2^2+k_3^2)/2+\beta,   ~~ \text{for}~d=3,
\end{align*}
where $k_1=k_2=k_3=1$. We choose the spatial domains as $\mathcal D=[0,2\pi]^d \ (d=2,3)$  and fix the Fourier node $32\times 32$ for $d=2$ and $32\times 32\times 32$ for $d=3$ respectively such that the spatial discretization errors are negligible. In addition, the convergence rate is obtained by the following formula
\begin{align*}
\quad {\rm  Rate}={\rm ln} \big(error_1/error_2)/{\rm ln} (\delta_1/\delta_2),
\end{align*}
where $\delta_l,error_l \ (l=1,2)$ are step sizes and errors with step size $\delta_l$, respectively. Moreover, the relative errors of discrete mass, Hamiltonian energy and quadratic energy on time level $t_n$ will be calculated by
\begin{align*}
 RM^n:=|(M^n- M^0)/M^0|, \  RH^n:=|(H^n- H^0)/H^0|, \ RE^n:=|(E^n- E^0)/E^0|, \ \ n=1,\cdots,N,
\end{align*}
respectively.

We first choose different $\beta$ to test the temporal accuracy in 2D/3D, and the results are summarized in Table~\ref{tab1}. As is shown that the ESAV-RK4 and ESAV-RK6 methods arrive at fourth-order and sixth-order convergence rates in time, respectively. Furthermore, for a fixed  time step and mesh size, the numerical errors are observed
to increase along with the growth of $\beta$. In this case, the high-order accurate numerical algorithms are more preferable in practical computations to
obtain a given high accuracy, especially in long-time simulation.

\vspace{-0.5cm}
\begin{table}[H]
\footnotesize
\caption{\footnotesize Temporal errors of the numerical solutions at $T=9$. \label{tab1}}
\centering
\begin{tabular}{ccccccc cccc }
\toprule
  & & &  2D case     &             &    &    &  3D case&    &   &       \\
\cmidrule(lr){4-7} \cmidrule(lr){8-11}
  & & &  $\tau\!=\!0.03$ &   $\tau\!=\!0.02$ & $\tau\!=\!0.015$  &  $ \tau\!=\!0.01$ &  $\tau\!=\!0.05$ &   $\tau\!=\!0.04$ & $ \tau\!=\!0.025$  &  $ \tau\!=\!0.0125$    \\

\hline
                & $\beta\!=\!5$  &$\lVert e \rVert_{\infty}$  & 3.16e-05 &  6.25e-06 & 1.98e-06 & 3.91e-07&2.43e-04&9.98e-05&1.52e-05&9.53e-07        \\
                &            &{\small Rate}           &* & 4.00 &4.00 & 4.00 & * &3.99 & 4.00 & 4.00     \\
 ESAV-RK4      & $\beta\!=\!6$  &$\lVert e \rVert_{\infty}$  &7.86e-05 & 1.55e-05 & 4.92e-06 &9.72e-07 &6.04e-04&2.48e--04&3.79e-05&2.37e-06   \\
            &            &{\small Rate}            & *&4.00 & 4.00& 4.00 &*&3.99& 4.00& 4.00 \\
                &$\beta\!=7\!$   &$\lVert e \rVert_{\infty}$  & 1.70e-04&3.36e-05 & 1.06e-05 &2.10e-06&1.30e-03&5.35e-04&8.19e-05&5.13e-06   \\
                &            &{\small Rate}            & *&4.00 & 4.00& 4.00&*&3.98&3.99&4.00  \\
 \hline
                &$\beta\!=\!5$   &$\lVert e \rVert_{\infty}$  & 5.08e-09 & 4.46e-10  & 7.95e-11  & 6.89e-12&1.09e-07&2.85e-08&1.70e-09&2.64e-11 \\
                &            &{\small Rate}           & *  & 5.99 &   5.99    &  6.03 &*&6.00&6.00&6.01  \\
    ESAV-RK6   &$\beta\!=\!6$   &$\lVert e \rVert_{\infty}$  & 1.82e-08 & 1.60e-09   & 2.85e-10  & 2.50e-11 & 3.89e-07 & 1.02e-07 &6.10e-09&9.54e-11\\
            &            &{\small Rate}            &* &6.00 & 6.00 & 6.00 &*&6.00&6.00 &6.00  \\
                &$\beta\!=\!7$       &$\lVert e \rVert_{\infty}$  & 5.35e-08 & 4.70e-09  & 8.37e-10  & 7.35e-11&1.14e-06&3.00e-07&1.79e-08&2.82e-10  \\
                &            &{\small Rate}            &* & 6.00  &  6.00  & 6.00&*& 6.00  &  6.00  & 6.00 \\
\toprule
\end{tabular}
\end{table}
\vspace{-1mm}
Moreover, we research the long-time behavior of the proposed schemes at a large time period $T=20$ with $\tau=0.01$ and the Fourier node $32\times 32$ for $d=2$ and $32\times 32\times 32$ for $d=3$. As is illustrated in Figure~\ref{fig1.1} (a)-(d) that the proposed schemes preserve
the discrete mass and energy exactly, which conforms the preceding theoretical analysis.

\begin{figure}[H]
\centering
\subfigure[]{
\begin{minipage}{0.2\linewidth}
\centerline{\includegraphics[height=3.4cm,width=1.2\textwidth]{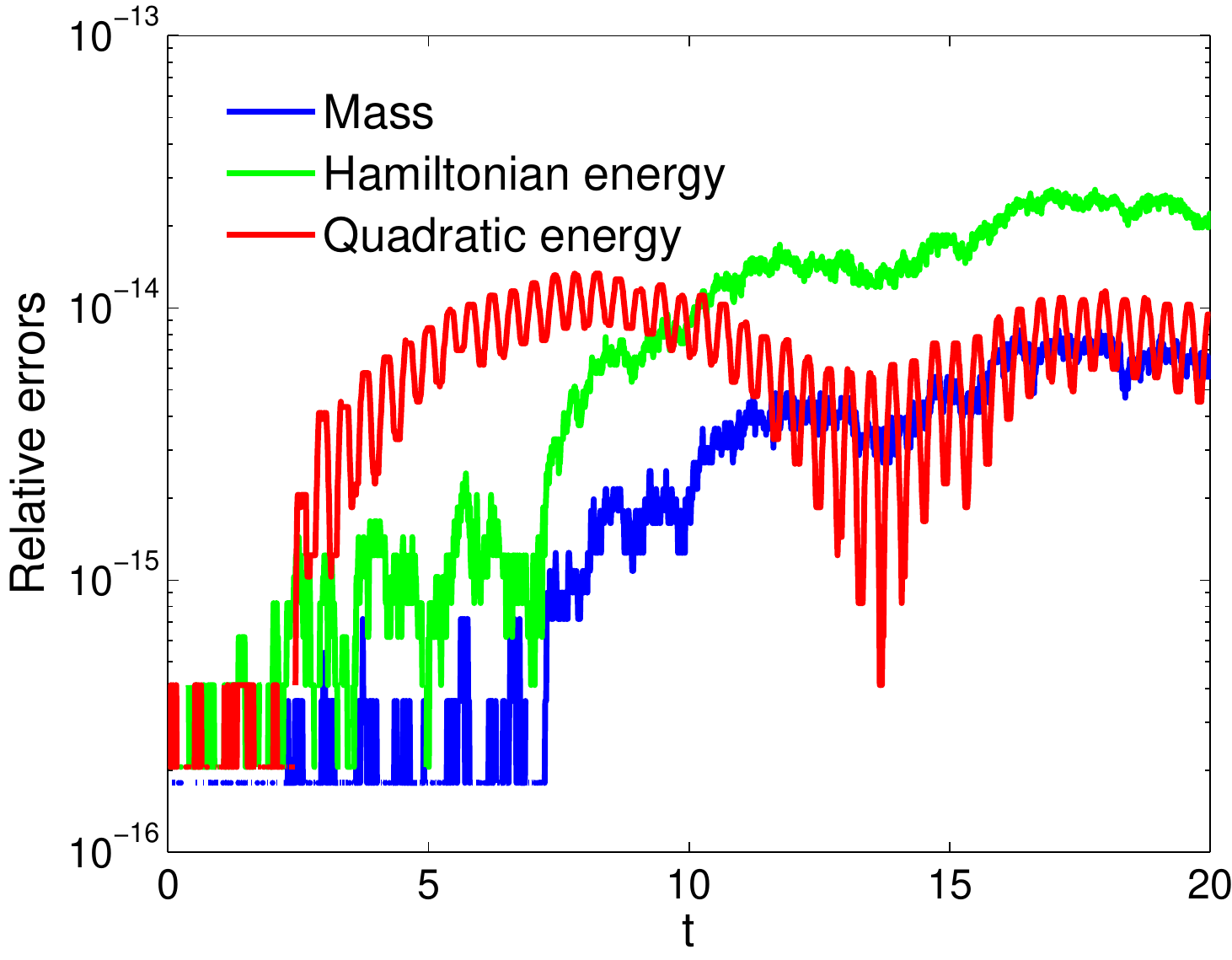}~~~}
\end{minipage}
}~~~
\subfigure[]{
\begin{minipage}{0.2\linewidth}
  \centerline{\includegraphics[height=3.4cm,width=1.2\textwidth]{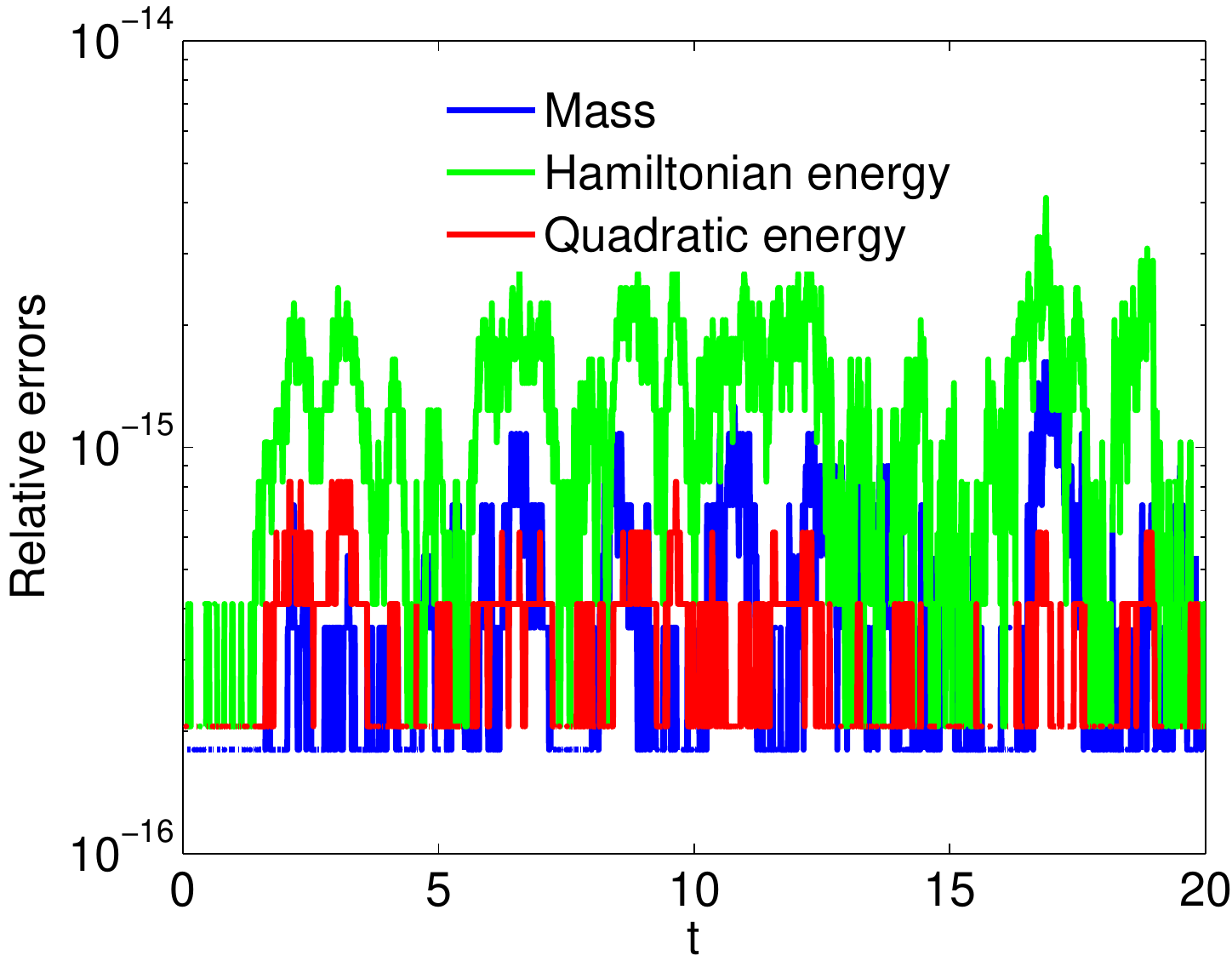}~~~}
\end{minipage}
}~~~
\subfigure[]{
\begin{minipage}{0.2\linewidth}
  \centerline{\includegraphics[height=3.4cm,width=1.2\textwidth]{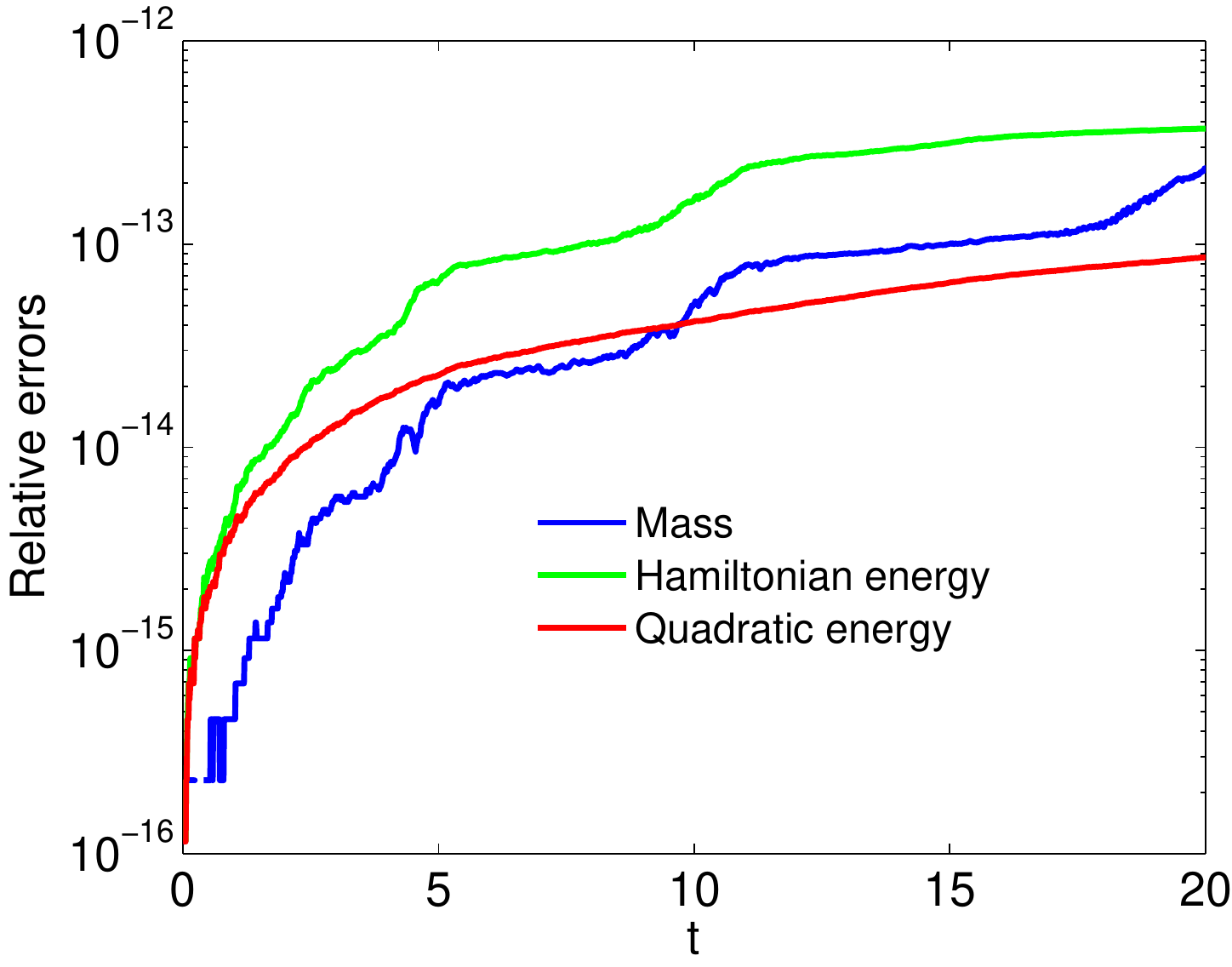}~~~}
\end{minipage}
}~~~
\subfigure[]{
\begin{minipage}{0.2\linewidth}
  \centerline{\includegraphics[height=3.4cm,width=1.2\textwidth]{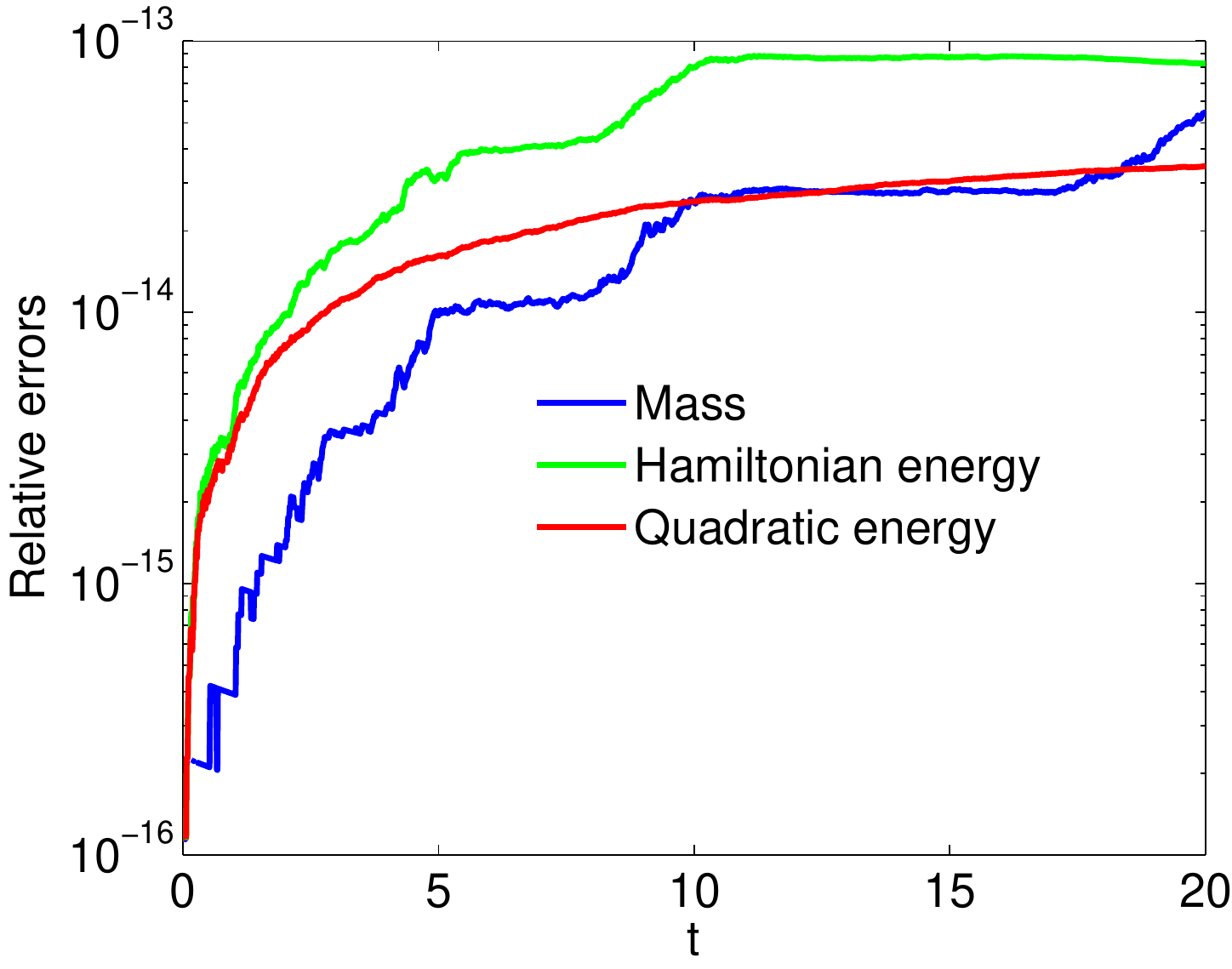}~~~}
\end{minipage}
}
\setlength{\abovecaptionskip}{-0.5mm}
\caption{\footnotesize Relative errors of discrete mass and energy with $\beta=5$ computed by (a) ESAV-RK4 for 2D case;
(b) ESAV-RK6 for 2D case; (c) ESAV-RK4 for 3D case; (d) ESAV-RK6 for 3D case, respectively.}
\label{fig1.1}
\end{figure}

\section{Conclusions}
In this paper, we present a novel class of arbitrary high-order exponential Runge-Kutta methods for solving the NLS equation by combing the SAV approach with the Lawson method. We show that the proposed method can preserve both the mass and the modified energy. Numerical tests are indicated to verify the accuracy and
effectiveness of the proposed schemes. The numerical strategy adopted in this paper can be generalized for general Hamiltonian partial differential systems to develop high-order
energy-preserving exponential Runge-Kutta methods. Here, we should note that, in general,  the particularly
interesting types of ERK methods are integrating factor (IF) methods and exponential time differencing (ETD) methods, respectively. The proposed method of this paper is actually assigned to the IF methods and arbitrary high-order structure-preserving ETD methods for the conservative systems will be presented in a separated report.

\section*{Acknowledgments}
\indent \normalsize Jin Cui's work is supported by Natural Research Fund of Nanjing Vocational College of Information Technology (Grant No. YK20200901).
Chaolong Jiang's work is partially supported by the National Natural Science Foundation of China (Grant No. 11901513), the Yunnan Provincial Department
of Education Science Research Fund Project (Grant No. 2019J0956) and the Science and Technology Innovation Team on Applied Mathematics in Universities of
 Yunnan. Yushun Wang's work is partially supported by the National Natural Science Foundation of China (Grant No. 11771213).


\end{document}